\documentclass[12pt]{article}

\usepackage{dsfont}
\usepackage{amsfonts,amsmath,amsthm}
\usepackage{algorithm, algorithmic}
\usepackage{color}
\usepackage{graphicx}
\usepackage{stmaryrd}        

\usepackage[paper=a4paper,dvips,top=2cm,left=2cm,right=2cm,
    foot=1cm,bottom=4cm]{geometry}

\begin{document}

\title{Eigenvalues and Singular Values of Dual Quaternion Matrices}
\author{ Liqun Qi\footnote{%
    Department of Applied Mathematics, The Hong Kong Polytechnic University, Hung Hom,
    Kowloon, Hong Kong
    ({\tt maqilq@polyu.edu.hk}).}
     \and and \
    Ziyan Luo\footnote{Corresponding author, Department of Mathematics,
  Beijing Jiaotong University, Beijing 100044, China. ({\tt zyluo@bjtu.edu.cn}). This author's work was supported by Beijing Natural Science Foundation (Grant No. Z190002).}
}
\date{\today}
\maketitle

\begin{abstract}
The poses of $m$ robotics in $n$ time points may be represented by an $m \times n$ dual quaternion matrix.    In this paper, we study the spectral theory of dual quaternion matrices.  We introduce right and left eigenvalues for square dual quaternion matrices.  If a right eigenvalue is a dual number, then it is also a left eigenvalue.  In this case, this dual number is called an eigenvalue of that dual quaternion matrix.  We show that the right eigenvalues of a dual quaternion Hermitian matrix are dual numbers.  Thus, they are eigenvalues.   An $n \times n$ dual quaternion Hermitian matrix is shown to have exactly $n$ eigenvalues.   It is positive semidefinite, or positive definite, if and only if all of its eigenvalues  are nonnegative, or positive and appreciable, dual numbers, respectively.   We present a unitary decomposition of a dual quaternion Hermitian matrix, and the singular value decomposition for a general dual quaternion matrix. The singular values of a dual quaternion matrix are nonnegative dual numbers.

\medskip



\end{abstract}

\renewcommand{\Re}{\mathds{R}}
\newcommand{\rank}{\mathrm{rank}}
\renewcommand{\span}{\mathrm{span}}
\newcommand{\X}{\mathcal{X}}
\newcommand{\A}{\mathcal{A}}
\newcommand{\I}{\mathcal{I}}
\newcommand{\B}{\mathcal{B}}
\newcommand{\C}{\mathcal{C}}
\newcommand{\OO}{\mathcal{O}}
\newcommand{\e}{\mathbf{e}}
\newcommand{\0}{\mathbf{0}}
\newcommand{\dd}{\mathbf{d}}
\newcommand{\ii}{\mathbf{i}}
\newcommand{\jj}{\mathbf{j}}
\newcommand{\kk}{\mathbf{k}}
\newcommand{\va}{\mathbf{a}}
\newcommand{\vb}{\mathbf{b}}
\newcommand{\vc}{\mathbf{c}}
\newcommand{\vg}{\mathbf{g}}
\newcommand{\vr}{\mathbf{r}}
\newcommand{\vt}{\rm{vec}}
\newcommand{\vx}{\mathbf{x}}
\newcommand{\vy}{\mathbf{y}}
\newcommand{\vu}{\mathbf{u}}
\newcommand{\vv}{\mathbf{v}}
\newcommand{\y}{\mathbf{y}}
\newcommand{\vz}{\mathbf{z}}
\newcommand{\T}{\top}

\newtheorem{Thm}{Theorem}[section]
\newtheorem{Def}[Thm]{Definition}
\newtheorem{Ass}[Thm]{Assumption}
\newtheorem{Lem}[Thm]{Lemma}
\newtheorem{Prop}[Thm]{Proposition}
\newtheorem{Cor}[Thm]{Corollary}
\newtheorem{example}[Thm]{Example}
\newtheorem{remark}[Thm]{Remark}

\section{Introduction}

Dual quaternions were introduced by
Clifford \cite{Cl73} in 1873.   They now have found wide applications in robotics and computer graphics \cite{BK20, BLH19, CKJC16, Da99, HY03, WYL12}.

Suppose that we have $m$ robotics.   The pose of each robotics $j$ is represented by a unit dual quaternion $a_{ij}$ \cite{BK20} at time $i$.   Thus, the poses of these $m$ robotics in time points $1, \cdots, n$ are represented by a dual quaternion matrix $A$.   If some entries of $A$ cannot be observed, how can we recover these entries with an approximate $m \times n$ dual quaternion matrix $B$ of a low ``rank''?   Unlike real matrices, or complex matrices, or even quaternion matrices \cite{Ro14, WLZZ18, Zh97}, which have matured matrix theories, there is a blank on the dual quaternion matrix research.  In this paper, we try to explore on the spectral theory of dual quaternion matrices.

In the next section, we present some preliminary knowledge on dual numbers, quaternions and dual quaternions.   In particular, a total order for dual numbers introduced in \cite{QL21a}, is recalled.

In Section 3, we introduce right and left eigenvalues for a square dual quaternion matrix.  If a right eigenvalue is a dual number, then it is also a left eigenvalue.   In that case, it is simply called an eigenvalue.  We show that the standard part $\lambda_{st}$ of a right eigenvalue $\lambda$ of a square dual quaternion matrix $A$ must be a right eigenvalue of the standard part $A_{st}$ of $A$.  Then we show that a right eigenvalue $\lambda$ of a dual quaternion Hermitian matrix $A$ must be a dual number, and present formulas of $\lambda$, its standard part $\lambda_{st}$ and infinitesimal part $\lambda_\I$.    Thus, a right eigenvalue of a dual quaternion Hermitian matrix $A$ is an eigenvalue of $A$.  We further show that if $\lambda_{st}$ is a simple right eigenvalue of $A_{st}$, then with formula of $\lambda_\I$, $\lambda = \lambda_{st} + \lambda_\I \epsilon$ is an eigenvalue of $A$.  Then we show that two eigenvectors of a dual quaternion Hermitian matrix, associated with two eigenvalues with distinct standard parts, are orthogonal to each other.

We present a unitary decomposition of a dual quaternion Hermitian matrix in Section 4.  A dual quaternion Hermitian matrix can always be diagonalized by a unitary decomposition.  Thus, an $n \times n$ dual quaternion Hermitian matrix has exactly $n$ eigenvalues. It is positive semidefinite, or positive definite, if and only if all of its eigenvalues are nonnegative, or positive and appreciable, dual numbers, respectively.

To establish the singular value decomposition for a general $m \times n$ dual quaternion matrix $A$, we need to decompose the $n \times n$ positive semidefinite dual quaternion Hermitian matrix $M = A^*A$ to the form of $U\Sigma^2U^*$, i.e., $M = L^2$ for $L = U\Sigma U^*$, where $U$ is a unitary $n \times n$ dual quaternion matrix, $\Sigma$ is an $n \times n$ block diagonal dual quaternion matrix.   This is always possible for a real, complex or quaternion positive semidefinite Hermitian matrix $M$.  But it is not always possible for a positive semidefinite dual quaternion Hermitian matrix $M$.   However, it is still possible for $M = A^*A$ for a general $m \times n$ dual quaternion matrix $A$.   Hence, we call such a positive semidefinite dual quaternion Hermitian matrix $M$ a perfect Hermitian matrix, and show that $M = A^*A$ is a  perfect Hermitian matrix in Section 5.

In Section 6, we present the singular value decomposition for a general dual quaternion matrix.  The singular values of a dual quaternion matrix are nonnegative dual numbers. The rank and the appreciable rank of a general dual quaternion matrix are introduced in terms of the number of positive singular values and the number of those with positive standard parts.

Some final remarks are made in Section 7.

We denote scalars, vectors and matrices by small letters, bold small letters and capital letters, respectively.

\section{Preliminaries}

\subsection{Dual Numbers}

Denote $\mathbb R$ and $\mathbb D$ as the set of the real numbers, and the set of the dual numbers, respectively.   A dual number $q$ has the form $q = q_{st} + q_\I\epsilon$, where $q_{st}$ and $q_\I$ are real numbers,  and $\epsilon$ is the infinitesimal unit, satisfying $\epsilon^2 = 0$.   We call $q_{st}$ the real part or the standard part of $q$, and $q_\I$ the dual part or the infinitesimal part of $q$.  The infinitesimal unit $\epsilon$ is commutative in multiplication with real numbers, complex numbers and quaternion numbers.  The dual numbers form a commutative algebra of dimension two over the reals.    If $q_{st} = 0$, we say that $q$ is appreciable, otherwise, we say that $q$ is infinitesimal.

In \cite{QL21a}, a total order was introduced for dual numbers.  Given two dual numbers $p, q \in \mathbb D$, $p = p_{st} + p_\I\epsilon$, $q = q_{st} + q_\I\epsilon$, where $p_{st}$, $p_\I$, $q_{st}$ and $q_\I$ are real numbers, we say that $p \le q$, if either $p_{st} < q_{st}$, or $p_{st} = q_{st}$ and $p_\I \le q_\I$.  In particular, we say that $p$ is positive, nonnegative, nonpositive or negative, respectively, if $p > 0$, $p \ge 0$, $p \le 0$ or $p < 0$, respectively.


\subsection{Quaternions}

Denote $\mathbb Q$ as the set of the quaternions.
A quaternion $q$ has the form
$q = q_0 + q_1\ii + q_2\jj + q_3\kk,$
where $q_0, q_1, q_2$ and $q_3$ are real numbers, $\ii, \jj$ and $\kk$ are three imaginary units of quaternions, satisfying
$\ii^2 = \jj^2 = \kk^2 =\ii\jj\kk = -1,$
$\ii\jj = -\jj\ii = \kk, \ \jj\kk = - \kk\jj = \ii, \kk\ii = -\ii\kk = \jj.$
The real part of $q$ is Re$(q) = q_0$.   The imaginary part of $q$ is Im$(q) = q_1\ii + q_2\jj +q_3\kk$.
A quaternion is called imaginary if its real part is zero.
The multiplication of quaternions satisfies the distribution law, but is noncommutative.

The conjugate of $q = q_0 + q_1\ii + q_2\jj + q_3\kk$ is
$\bar q \equiv q^* = q_0 - q_1\ii - q_2\jj - q_3\kk.$
The magnitude of $q$ is
$|q| = \sqrt{q_0^2+q_1^2+q_2^2+q_3^2}.$
It follows that the inverse of a nonzero quaternion $q$ is given by
$q^{-1} = {q^* / |q|^2}.$ For any two quaternions $p$ and $q$, we have $(pq)^* = q^*p^*$.

Two quaternions $p$ and $q$ are said to be similar if there is a nonzero quaternion $u$ such that $p = u^{-1}qu$.   We denote $p \sim q$.    It is easy to check that $\sim$ is an equivalence relation on the quaternions.   Denote by $[q]$ the equivalence class containing $q$. Then $[q]$ is a singleton if and only if $q$ is a real number.

Denote the collection of $n$-dimensional quaternion  vectors by ${\mathbb {Q}}^n$. For $\vx = (x_1, x_2,\cdots, x_n)^\top, \vy = (y_1, y_2,\cdots, y_n)^\top  \in {\mathbb {Q}}^n$, define $\vx^*\vy = \sum_{i=1}^n x_i^*y_i$, where $\vx^* = (x_1^*, x_2^*,\cdots, x_n^*)$ is the conjugate transpose of $\vx$.  Also, denote $\bar \vx = (x_1^*, x_2^*,\cdots, x_n^*)^\top$, the conjugate of $\vx$.

\subsection{Quaternion Matrices}

Denote the collections of complex and quaternion $m \times n$ matrices by ${\mathbb C}^{m \times n}$ and ${\mathbb Q}^{m \times n}$, respectively.
Denote a quaternion matrix $A= (a_{ij}) \in {\mathbb Q}^{m \times n}$ as
$A = A_0 + A_1\ii + A_2\jj + A_3\kk,$
where $A_0, A_1, A_2, A_3 \in {\mathbb R}^{m \times n}$.   The transpose of $A$ is denoted as $A^\top = (a_{ji})$. The conjugate of $A$ is denoted as $\bar A = (a_{ij}^*)$.   The conjugate transpose of $A$ is denoted as $A^* = (a_{ji}^*) = \bar A^\top$.


Let $A \in {\mathbb Q}^{m \times n}$ and $B \in {\mathbb Q}^{n \times r}$.   Then we have $(AB)^* = B^*A^*$.   In general, $(AB)^\top \not = B^\top A^\top$ and $\overline {AB} \not = \bar A \bar B$.

A square quaternion matrix $A \in {\mathbb Q}^{n \times n}$ is called normal if $A^*A = AA^*$, Hermitian if $A^* = A$; unitary if $A^*A = I$; skew-Hermitian if $A^* = -A$; and invertible (nonsingular) if there exists $AB = BA = I$ for some $B \in {\mathbb Q}^{n \times n}$.  In that case, we denote $A^{-1} = B$.

We have $(AB)^{-1} = B^{-1}A^{-1}$ if $A$ and $B$ are invertible, and $\left(A^*\right)^{-1} = \left(A^{-1}\right)^*$ if $A$ is invertible.

A quaternion Hermitian matrix $A$ is called positive semidefinite if for any $\vx \in {\mathbb Q}^n$, $\vx^*A\vx \ge 0$; $A$ is called positive definite if for any $\vx \in {\mathbb Q}^n$ with $\vx \not = \0$,  we have $\vx^*A\vx > 0$.  A square quaternion matrix $A \in {\mathbb Q}^{n \times n}$ is unitary if and only if its column (row) vectors form an orthonormal basis of ${\mathbb Q}^n$.

Suppose that $A \in {\mathbb Q}^{n \times n}$, $\vx \in {\mathbb Q}^n$, $\vx \not = \0$, and $\lambda \in
{\mathbb Q}$, satisfy
$A\vx = \vx\lambda$.
Then $\lambda$ is called a right eigenvalue of $A$, with $\vx$ as an associated right eigenvector.  If $\lambda$ is complex and its imaginary part is nonnegative, then $\lambda$ is called a standard right eigenvalue of $A$. On the other hand, if there are $\vx \in {\mathbb Q}^n$, $\vx \not = \0$, and $\lambda \in
{\mathbb Q}$, satisfying
$A\vx = \lambda\vx$, then $\lambda$ is called a left eigenvalue of $A$, with $\vx$ as an associated left eigenvector.  If $\lambda$ is real and a right eigenvalue of $A$, then it is also a left eigenvalue of $A$, as a real number is commutative with a quaternion vector.

  The following theorem collects some known  results of the right eigenvalues of quaternion matrices \cite{Ro14, WLZZ18, Zh97}.

\begin{Thm} \label{t2.1}
Suppose that $A \in {\mathbb Q}^{n \times n}$.   Then we have the following properties of the right eigenvalues of $A$.

1. If $\lambda \in {\mathbb Q}$ is a right eigenvalue of $A$, then $\mu \in [\lambda]$ is also a right eigenvalue of $A$.

2.  If $\lambda \in {\mathbb C}$ is a right eigenvalue of $A$, then $\bar \lambda$ is also a right eigenvalue of $A$ with the same right eigenvectors.

3. $A$ has exactly $n$ standard right eigenvalues.

4. $A$ is normal if and only if there is a unitary matrix $U \in {\mathbb Q}^{n \times n}$ such that
\begin{equation} \label{e2}
U^*AU = D = {\rm diag}(\lambda_1, \lambda_2,\cdots, \lambda_n),
\end{equation}
where $\lambda_1, \lambda_2,\cdots, \lambda_n$ are right eigenvalues of $A$.

5. $A$ is Hermitian if and only if $A$ is normal, and all the right eigenvalues of $A$ are real.

6. If $A$ is Hermitian, then $A$ has exactly $n$ real right eigenvalues, and there is a unitary matrix $U \in {\mathbb Q}^{n \times n}$ such that (\ref{e2}) holds with $\lambda_1 \ge \lambda_2\ge \cdots \ge \lambda_n$ as the real right eigenvalues of $A$; $A$ is positive semidefinite if and only if $\lambda_n \ge 0$; $A$ is positive definite if and only if $\lambda_n > 0$.


\end{Thm}

For any $\vx \in {\mathbb Q}^n$, if $\vx \not = \0$, then $\vx^*\vx$ is a real positive number.   It is commutative with a quaternion.   Thus, for any $A \in {\mathbb Q}^{n \times n}$ and $\vx \in {\mathbb Q}^n$, $\vx \not = \0$, we have
$$(\vx^*\vx)^{-1}(\vx^*A\vx) = (\vx^*A\vx)(\vx^*\vx)^{-1},$$
and we may denote
$${\vx^*A\vx \over \vx^*\vx} := (\vx^*\vx)^{-1}(\vx^*A\vx) = (\vx^*A\vx)(\vx^*\vx)^{-1}.$$
Thus, we have the following proposition.

\begin{Prop} \label{p2.2}
Suppose that $\lambda \in \mathbb Q$ is a right eigenvalue of $A \in {\mathbb Q}^{n \times n}$, with associated right eigenvector $\vx \in {\mathbb Q}^n$. Then
$$\lambda = {\vx^*A\vx \over \vx^*\vx}.$$
\end{Prop}






\subsection{Dual Quaternions}

Denote the set of dual quaternions as $\mathbb {DQ}$.   A dual quaternion $q \in \mathbb {DQ}$ has the form
$$q = q_{st} + q_\I\epsilon,$$
where $q_{st}, q_\I \in \mathbb {Q}$ are the standard part and the infinitesimal part of $q$, respectively. If $q_{st} \not = 0$, then we say that $q$ is appreciable.   We can derive that $q$ is invertible if and only if  $q$ is appreciable. In this case, we have
$$q^{-1} = q_{st}^{-1} - q_{st}^{-1}q_\I q_{st}^{-1} \epsilon.$$


The conjugate of $q$ is
$$q^* = q_{st}^* + q_\I^*\epsilon.$$
By this, if $q = q^*$, then $q$ is a dual number.

Denote the collection of $n$-dimensional dual quaternion  vectors by ${\mathbb {DQ}}^n$.

For $\vx = (x_1, x_2,\cdots, x_n)^\top$, $\vy = (y_1, y_2,\cdots, y_n)^\top  \in {\mathbb {DQ}}^n$, define $\vx^*\vy = \sum_{i=1}^n x_i^*y_i$, where $\vx^* = (x_1^*, x_2^*,\cdots, x_n^*)$ is the conjugate transpose of $\vx$.   We say $\vx$ is appreciable if at least one of its component is appreciable.  We also say that $\vx$ and $\vy$ are orthogonal to each other if $\vx^*\vy = 0$.  By \cite{QL21a}, for any $\vx \in {\mathbb {DQ}}^n$, $\vx^*\vx$ is a nonnegative dual number, and if $\vx$ is appreciable, $\vx^*\vx$ is a positive dual number.

Note that a dual number is commutative with a dual quaternion vector.

\section{Eigenvalues of Dual Quaternion Matrices}

Denote the collections of $m \times n$ dual quaternion  matrices by ${\mathbb {DQ}}^{m \times n}$.  Then $A \in {\mathbb {DQ}}^{m \times n}$ can be written as
$$A = A_{st} + A_\I \epsilon,$$
where $A_{st}, A_\I \in {\mathbb {Q}}^{m \times n}$ are the standard part and the infinitesimal part of $A$, respectively. Note that for a dual quaternion matrix $A=A_{st}+A_{\I}\epsilon$, if $A_{st} = O$, the analysis for $A$ will be analogous to that for the quaternion matrix $A_\I$. Thus, unless otherwise stated, we will assume that the dual quaternion matrix $A$ is appreciable throughout the paper, i.e., $A_{st}\neq O$.

The conjugate transpose of $A$ is
$$A^* = A_{st}^* + A_\I^* \epsilon.$$

A square dual quaternion matrix $A \in {\mathbb {DQ}}^{n \times n}$ is called normal if $A^*A = AA^*$, Hermitian if $A^* = A$; unitary if $A^*A = I$; and invertible (nonsingular) if $AB = BA = I$ for some $B \in {\mathbb {DQ}}^{n \times n}$.  In that case, we denote $A^{-1} = B$.

We have $(AB)^{-1} = B^{-1}A^{-1}$ if $A$ and $B$ are invertible, and $\left(A^*\right)^{-1} = \left(A^{-1}\right)^*$ if $A$ is invertible.

Suppose that $A \in {\mathbb {DQ}}^{n \times n}$ is an Hermitian matrix.  For any $\vx \in {\mathbb {DQ}}^n$, we have
$$(\vx^*A\vx)^* = \vx^*A\vx.$$
This implies that $\vx^*A\vx$ is a dual number.  With the total order of dual numbers defined in Section 2, we may define positive semidefiniteness and positive definiteness of dual quaternion Hermitian matrices.  A dual quaternion Hermitian matrix $A \in {\mathbb {DQ}}^{n \times n}$ is called positive semidefinite if for any $\vx \in {\mathbb {DQ}}^n$, $\vx^*A\vx \ge 0$; $A$ is called positive definite if for any $\vx \in {\mathbb {DQ}}^n$ with $\vx$ being appreciable,  we have $\vx^*A\vx > 0$ and is appreciable.
A square quaternion matrix $A \in {\mathbb {DQ}}^{n \times n}$ is unitary if and only if its column (row) vectors form an orthonormal basis of ${\mathbb {DQ}}^n$.  Let $2 \le k \le n$.  We say that $A \in {\mathbb {DQ}}^{n \times k}$ is partially unitary if its column vectors are unit vectors and orthogonal to each other.

Then, $A$ is a dual quaternion Hermitian matrix if and only if $A_{st}$ and $A_\I$ are two quaternion Hermitian matrices.


Suppose that $A \in {\mathbb {DQ}}^{n \times n}$.  If there are $\lambda \in \mathbb {DQ}$, $\vx \in {\mathbb {DQ}}^n$, where $\vx$ is appreciable, such that
\begin{equation}
A\vx = \vx\lambda,
\end{equation}
then we say that $\lambda$ is a right eigenvalue of $A$, with $\vx$ as an associated right eigenvector.
If there are $\lambda \in \mathbb {DQ}$, $\vx \in {\mathbb {DQ}}^n$, where $\vx$ is appreciable, such that
$$A\vx = \lambda\vx,$$
then we say that $\lambda$ is a left eigenvalue of $A$, with $\vx$ as an associated left eigenvector.   If $\lambda$ is a dual number and a right eigenvalue of $A$, then it is also a left eigenvalue of $A$, as  a dual number is commutative with a dual quaternion vector.   In this case, we simply say that it is an eigenvalue of $A$, with $\vx$ as an associated eigenvector.

For any $\vx \in {\mathbb {DQ}}^n$, if $\vx$ is appreciable, then $\vx^*\vx$ is an appreciable positive dual number.   It is commutative with a dual quaternion.   Thus, for any $A \in {\mathbb {DQ}}^{n \times n}$ and $\vx \in {\mathbb {DQ}}^n$, $\vx$ appreciable, we have
$$(\vx^*\vx)^{-1}(\vx^*A\vx) = (\vx^*A\vx)(\vx^*\vx)^{-1},$$
and we may denote
$${\vx^*A\vx \over \vx^*\vx} := (\vx^*\vx)^{-1}(\vx^*A\vx) = (\vx^*A\vx)(\vx^*\vx)^{-1}.$$
Similar to Proposition \ref{p2.2}, we have the following proposition.

\begin{Prop} \label{p3.0}
Suppose that $\lambda \in \mathbb {DQ}$ is a right eigenvalue of $A \in {\mathbb {DQ}}^{n \times n}$, with associated right eigenvector $\vx \in {\mathbb {DQ}}^n$. Then
\begin{equation} \label{e5}
\lambda = {\vx^* A\vx \over \vx^*\vx}.
\end{equation}
\end{Prop}

Suppose that $\lambda$ is a right eigenvalue of $A \in {\mathbb {DQ}}^{n \times n}$, with $\vx$ as an associated right eigenvector.   If $q \in \mathbb {DQ}$ is appreciable, then we have $A\vx = \vx\lambda$ implies $A(\vx q) = (A\vx)q = \vx \lambda q = (\vx q)(q^{-1}\lambda q)$.   Thus, if $\lambda$ is a right eigenvalue of $A$, then any dual quaternion in $[\lambda]$ is also a right eigenvalue of $A$.

We have the following theorem.

\begin{Thm} \label{t3.1}
Suppose that $A = A_{st}+A_\I\epsilon \in {\mathbb {DQ}}^{n \times n}$.   Then $\lambda = \lambda_{st} + \lambda_\I \epsilon$ is a right eigenvalue of $A$ with a right eigenvector $\vx_{st}+\vx_I\epsilon$ only if $\lambda_{st}$ is a right eigenvalue of the quaternion matrix $A_{st}$ with a
right eigenvector $\vx_{st}$, i.e., $\vx_{st} \not = \0$ and
\begin{equation} \label{e3}
A_{st}\vx_{st} = \vx_{st}\lambda_{st}.
\end{equation}
We have
\begin{equation} \label{e6}
\lambda_{st} = {\vx_{st}^* A_{st}\vx_{st} \over \vx_{st}^*\vx_{st}}.
\end{equation}

Furthermore, if $\lambda_{st}$ is a right eigenvalue of the quaternion matrix $A_{st}$ with a right
eigenvector $\vx_{st}$, then $\lambda$ is a right eigenvalue of $A$ with a right eigenvector $\vx$ if and only if $\lambda_\I$ and $\vx_\I$ satisfy
\begin{equation} \label{e4}
\vx_{st}\lambda_\I = A_\I\vx_{st} + A_{st}\vx_\I - \vx_\I \lambda_{st}.
\end{equation}
\end{Thm}
\begin{proof}
By definition, $\vx$ is a right eigenvalue of $A$ with a right eigenvector $\vx$ if and only if $A\vx = \vx \lambda$ and $\vx_{st} \not = \0$.   We may write $A\vx = \vx \lambda$ as
$$(A_{st}+A_\I\epsilon)(\vx_{st}+\vx_I\epsilon) = (\vx_{st}+\vx_I\epsilon)(\lambda_{st} + \lambda_\I \epsilon).$$
This is equivalent to (\ref{e3}) and (\ref{e4}).  By Proposition \ref{p2.2}, we have (\ref{e6}).
\end{proof}

The following theorem indicates that the right eigenvalues of a dual quaternion Hermitian matrix must be a dual number.

\begin{Thm} \label{t3.2}
A right eigenvalue $\lambda$ of an Hermitian matrix $A = A_{st} + A_\I\epsilon \in {\mathbb {DQ}}^{n \times n}$ must be a dual number, hence an eigenvalue of $A$, and its standard part $\lambda_{st}$ is a right eigenvalue of the quaternion Hermitian matrix $A_{st}$.   Furthermore, assume that $\lambda = \lambda_{st} + \lambda_\I \epsilon$,
$\vx = \vx_{st} + \vx_\I \epsilon \in {\mathbb {DQ}}^n$ is an eigenvector of $A$, associate with the eigenvalue $\lambda$,
where $\vx_{st}, \vx_\I \in {\mathbb Q}^n$.   Then we have
\begin{equation} \label{e7}
\lambda_\I = {\vx_{st}^* A_\I \vx_{st} \over \vx_{st}^*\vx_{st}}.
\end{equation}

A dual quaternion Hermitian matrix has at most $n$ dual number eigenvalues and no other right eigenvalues.

An eigenvalue of a positive semidefinite Hermitian matrix $A \in {\mathbb {DQ}}^{n \times n}$ must be a nonnegative dual number.   In that case, $A_{st}$ must be positive semidefinite.
An eigenvalue of a positive definite Hermitian matrix $A \in {\mathbb {DQ}}^{n \times n}$ must be an appreciable positive dual number.   In that case, $A_{st}$ must be positive definite.
\end{Thm}
\begin{proof}   Suppose that $A \in {\mathbb {DQ}}^{n \times n}$ is an Hermitian matrix, and $\lambda$ is a right eigenvalue of $A$.  Let $\vx$ be the corresponding right eigenvector.    Then we have $A\vx = \vx \lambda$, and $\vx$ is appreciable.    We have
\begin{equation} \label{e8}
\vx^*A\vx = \vx^*\vx \lambda.
\end{equation}
Substitute $\vx = \vx_{st} + \vx_\I \epsilon$, $A = A_{st} + A_\I \epsilon$, $\lambda = \lambda_{st} + \lambda_\I \epsilon$ to (\ref{e8}), and consider the infinitesimal part of the equality.  We have
$$(\vx_{\I}^* \vx_{st} + \vx_{st}^*\vx_\I )\lambda_{st} + \vx_{st}^*\vx_{st}\lambda_\I =
\vx_\I^*A_{st}\vx_{st} + \vx_{st}^*A_{st}\vx_\I + \vx_{st}^*A_\I \vx_{st}.$$
Since $A_{st}\vx_{st} = \vx_{st}\lambda_{st}$, $\vx_{st}^*A_{st} = \lambda_{st}\vx_{st}^*$ and $\lambda_{st}$ is a real number, the above equality reduces to
$$\vx_{st}^*\vx_{st}\lambda_\I = \vx_{st}^*A_\I \vx_{st}.$$
This proves (\ref{e7}). Hence, $\lambda$ is a dual number and an eigenvalue of $A$.

Now, the other conclusions follow from Theorems \ref{t2.1} and \ref{t3.1}.
\end{proof}

Furthermore, we have the following two theorems.

\begin{Thm} \label{t3.3}
Suppose that $A = A_{st} + A_\I\epsilon \in {\mathbb {DQ}}^{n \times n}$ is an Hermitian matrix, and $\lambda_{st}$ is a simple right eigenvalue of $A_{st}$ with a right eigenvector $\vx_{st}$.   Then $\lambda = \lambda_{st} + \lambda_\I \epsilon$ is an eigenvalue of $A$, where $\lambda_\I$ is calculated by (\ref{e7}).   Furthermore, if $A_{st}$ has n simple right eigenvalues $\lambda_{st,i}$'s with associated unit right eigenvectors $\vx_{st,i}$'s. Then $A$ has exactly $n$  eigenvalues
\begin{equation}
\lambda_i = \lambda_{st,i}+ \lambda_{\I,i}\epsilon,
\end{equation}
with associated eigenvectors $\vx_i = \vx_{st,i} + \vx_{\I,i}\epsilon$, $i=1,\cdots, n$, where
\begin{equation}\label{equation 10}
\lambda_{\I,i}= \vx_{st,i}^*A_\I\vx_{st,i}, ~~\vx_{\I, i} = \sum\limits_{j\neq i} \frac{\vx_{st,j}\vx_{st,j}^*(A_{\I}-\lambda_{\I,i} I_n )\vx_{st,i}}{\lambda_{st,i}-\lambda_{st,j}},~~i=1,\cdots,n
\end{equation}

\end{Thm}
\begin{proof}  Since $A$ is a dual quaternion Hermitian matrix, $A_{st}$ is a quaternion Hermitian matrix and $\lambda_{st}$ is real.
Then (\ref{e4}) becomes
\begin{equation} \label{e9}
A_\I \vx_{st} - \vx_{st}\lambda_\I = (\lambda_{st} I_n - A_{st})\vx_\I.
\end{equation}

Since $\lambda_{st}$ is a simple right eigenvalue of $A_{st}$, by Theorem \ref{t2.1} 4-5 and Theorem \ref{t3.1},  (\ref{e9}) has a solution $\vx_\I$ if and only if
$$\vx_{st}^*A_\I\vx_{st} - \vx_{st}^*\vx_{st}\lambda_\I= 0.$$
This is true by (\ref{e7}).

This implies (\ref{e9}) has a solution $\vx_\I$.   By Theorems \ref{t3.1} and \ref{t3.2}, $\lambda$ is an eigenvalue of $A$.

For the furthermore part, it suffices to show that $A_\I \vx_{st,i} - \vx_{st,i}\lambda_{\I,i} = (\lambda_{st,i} I_n - A_{st})\vx_{\I,i}$, $i=1,\cdots,n$. By Theorem \ref{t2.1}, we have
$$I_n = \sum\limits_{j=1}^n \vx_{st,j}\vx_{st,j}^*  ~\mathrm{and}~~ A_{st} = \sum\limits_{j=1}^n \lambda_{st,j}\vx_{st,j}\vx_{st,j}^*.$$
It follows readily that
\begin{eqnarray}\label{x_{I,i}}
& & (\lambda_{st,i} I_n - A_{st})\vx_{\I,i} \nonumber\\
& = & (\lambda_{st,i} I_n - A_{st}) \sum\limits_{j\neq i} \frac{\vx_{st,j}\vx_{st,j}^*(A_{\I}-\lambda_{\I,i} I_n )\vx_{st,i}}{\lambda_{st,i}-\lambda_{st,j}} \nonumber\\
& = & \left(\sum\limits_{j\neq i}\left(\lambda_{st,i}-\lambda_{st,j}\right) \vx_{st,j}\vx_{st,j}^*\right)\left(\sum\limits_{j\neq i} \frac{\vx_{st,j}\vx_{st,j}^*}{\lambda_{st,i}-\lambda_{st,j}}\right) (A_{\I}-\lambda_{\I,i} I_n )\vx_{st,i} \nonumber\\
& = & \left(I_n- \vx_{st,i}\vx_{st,i}^*\right)(A_{\I}-\lambda_{\I,i} I_n )\vx_{st,i} \nonumber\\
& = & A_\I \vx_{st,i} - \lambda_{\I,i} \vx_{st,i} -\vx_{st,i}\vx_{st,i}^*A_{\I}\vx_{st,i}+ \lambda_{\I,i} \vx_{st,i}\vx_{st,i}^*\vx_{st,i}  \nonumber\\
& = & A_\I \vx_{st,i} -\vx_{st,i}\vx_{st,i}^*A_{\I}\vx_{st,i} \nonumber\\
& = & A_\I \vx_{st,i} - \vx_{st,i}\lambda_{\I,i},\nonumber
\end{eqnarray}
where the fifth equality is due to $\vx_{st,i}^*\vx_{st,i}=1$, and the last equality is from the formula of $\lambda_{\I,i}$ as defined in \eqref{equation 10}.
This completes the proof.

\end{proof}

For eigenvectors of a dual quaternion Hermitian matrix, we have the following proposition.

\begin{Prop} \label{p3.4}
Two eigenvectors of an Hermitian matrix $A \in {\mathbb {DQ}}^{n \times n}$, associated with two  eigenvalues with distinct standard parts, are orthogonal to each other.
\end{Prop}
\begin{proof}
Suppose $\vx$ and $\vy$ are two  eigenvectors of an Hermitian matrix $A \in {\mathbb {DQ}}^{n \times n}$, associated with two eigenvalues $\lambda = \lambda_{st} + \lambda_\I \epsilon$ and $\mu = \mu_{st} + \mu_\I \epsilon$, respectively, and $\lambda_{st} \not = \mu_{st}$.   By Theorem \ref{t3.2}, $\lambda$ and $\mu$ are dual numbers.  They are commutative with a dual quaternion.
We have
$$\lambda(\vx^*\vy) = (\vx\lambda)^*\vy = (A\vx)^*\vy = \vx^*A\vy = \vx^*\vy\mu = \mu \vx^*\vy,$$
i.e.,
$$(\lambda - \mu )(\vx^*\vy) = 0.$$
Since $\lambda_{st} \not = \mu_{st}$, $(\lambda - \mu)^{-1}$ exists.   We have $\vx^*\vy = 0$.
\end{proof}

Note that for this proposition, only $\lambda \not = \mu$ is not enough.

If $A_{st}$ has multiple eigenvalues, then the situation is somewhat complicated.   We solve this problem by presenting a unitary decomposition for dual quaternion Hermitian matrices in the next section.

\section{Unitary Decomposition of Dual Quaternion Hermitian Matrices}

Now we present the unitary decomposition of a dual quaternion Hermitian matrix.

\begin{Thm} \label{t4.1}
Suppose that $A=A_{st}+A_\I \in {\mathbb {DQ}}^{n \times n}$ is an Hermitian matrix.  Then there are unitary matrix $U \in {\mathbb {DQ}}^{n \times n}$ and a diagonal matrix $\Sigma \in {\mathbb {D}}^{n \times n}$ such that $\Sigma = U^*AU$, where
\begin{equation} \label{eee1}
\Sigma\equiv {\rm diag}\left(\lambda_1+\lambda_{1,1}\epsilon,\cdots, \lambda_1+\lambda_{1,k_1}\epsilon, \lambda_2+\lambda_{2,1}\epsilon,\cdots, \lambda_r+\lambda_{r,k_r}\epsilon\right).
\end{equation}
with the diagonal entries of $\Sigma$ being $n$  eigenvalues of $A$,
\begin{equation} \label{eee2}
A\vu_{i, j} = \vu_{i, j}(\lambda_i +\lambda_{i, j}\epsilon),
\end{equation}
for $j = 1, \cdots, k_i$ and $i = 1, \cdots, r$, $U = (\vu_{1,1}, \cdots, \vu_{1, k_1}, \cdots, \vu_{r, k_r})$,
$\lambda_1 > \lambda_2 > \cdots > \lambda_r$ are real numbers, $\lambda_i$ is a $k_i$-multiple right eigenvalue of $A_{st}$, $\lambda_{i, 1} \ge \lambda_{i, 2} \ge \cdots \ge \lambda_{i, k_i}$ are also real numbers.   Counting possible multiplicities $\lambda_{i, j}$, the form $\Sigma$ is unique.


\end{Thm}
\begin{proof}   Let $A \in {\mathbb {DQ}}^{n \times n}$ be an Hermitian matrix.  Denote $A = A_{st} + A_\I \epsilon$, where $A_{st}, A_\I \in {\mathbb {Q}}^{n \times n}$.  Then $A_{st}$ and $A_\I$ are Hermitian.   This implies that there is a quaternion unitary matrix $S \in {\mathbb {Q}}^{n \times n}$ and a real diagonal matrix $D \in {\mathbb {R}}^{n \times n}$ such that $D = SA_{st}S^*$.
Suppose that $D = {\rm diag}(\lambda_1I_{k_1}, \lambda_2I_{k_2}, \cdots, \lambda_rI_{k_r})$, where $\lambda_1 > \lambda_2 > \cdots > \lambda_r$, and $I_{k_i}$ is a $k_i \times k_i$ identity matrix, and $\sum_{i=1}^r k_i = n$.  Let $M = SAS^*$.  Then
\begin{eqnarray*}
&& M \\ & = & D + SA_\I \epsilon S^*\\
& = & \begin{bmatrix}
\lambda_1I_{k_1} + \epsilon C_{11} & \epsilon C_{12} & \cdots & \epsilon C_{1r} \\
\epsilon C_{12}^* & \lambda_2I_{k_2} + \epsilon C_{22} & \cdots  & \epsilon C_{2r} \\
\vdots & \vdots & \ddots & \vdots \\
\epsilon C_{1r}^* & \epsilon C_{2r}^* &  \cdots & \lambda_r I_{k_r} + \epsilon C_{rr}
\end{bmatrix},
\end{eqnarray*}
where each $C_{ij}$ is a quaternion matrix of adequate dimensions, each $C_{ii}$ is Hermitian.

Let
$$P = \begin{bmatrix}
I_{k_1} & {\epsilon C_{12} \over \lambda_1-\lambda_2} & \cdots & {\epsilon C_{1r} \over \lambda_1-\lambda_r}\\
-{\epsilon C_{12}^* \over \lambda_1-\lambda_2} & I_{k_2} & \cdots  & {\epsilon C_{2r} \over \lambda_2-\lambda_r} \\
\vdots & \vdots & \ddots & \vdots \\
-{\epsilon C_{1r}^* \over \lambda_1-\lambda_r} & -{\epsilon C_{2r}^* \over \lambda_2-\lambda_r}  &  \cdots & I_{k_r}
\end{bmatrix}.$$
Then
$$P^* = \begin{bmatrix}
I_{k_1} & -{\epsilon C_{12} \over \lambda_1-\lambda_2} & \cdots & -{\epsilon C_{1r} \over \lambda_1-\lambda_r}\\
{\epsilon C_{12}^* \over \lambda_1-\lambda_2} & I_{k_2} & \cdots  & -{\epsilon C_{2r} \over \lambda_2-\lambda_r} \\
\vdots & \vdots & \ddots & \vdots \\
{\epsilon C_{1r}^* \over \lambda_1-\lambda_r} & {\epsilon C_{2r}^* \over \lambda_2-\lambda_r}  &  \cdots & I_{k_r}
\end{bmatrix}.$$
Direct calculations certify $PP^* = P^*P = I_n$ and
$$\Sigma'\equiv PMP^* =(PS)A(PS)^* = {\rm diag}(\lambda_1I_{k_1}+\epsilon C_{11}, \lambda_2I_{k_2} +\epsilon C_{22}, \cdots, \lambda_rI_{k_r}+\epsilon C_{rr}).$$
Since $P$ and $S$ are unitary matrices, then so is $PS$.  
Noting that each $C_{jj}$ is an Hermitian quaternion matrix, by applying Theorem \ref{t2.1}, we can find unitary matrices $U_1\in {\mathbb{Q}}^{k_1\times k_1}$, $\cdots$, $U_r\in {\mathbb{Q}}^{k_r\times k_r}$ that diagonalize $C_{11}$, $\cdots$, $C_{rr}$, respectively. That is, there exist real numbers $\lambda_{1,1} \ge \cdots \ge \lambda_{1,k_1}$, $\lambda_{2,1} \ge \cdots \ge \lambda_{2, k_2}$, $\cdots$, $\lambda_{r,1} \ge \cdots \ge \lambda_{r,k_r}$ such that
\begin{equation}\label{blocks}
U_j^* C_{jj} U_j = {\rm diag}\left(\lambda_{j,1}, \cdots, \lambda_{j,k_j}\right), ~~j=1,\cdots, r.
\end{equation}
Denote $V \equiv {\rm diag}\left(U_1, \cdots, U_r\right)$. We can easily verify that $V$ is unitary. Thus, $U\equiv (PS)^*V$ is also unitary.  Denote
$$\Sigma\equiv {\rm diag}\left(\lambda_1+\lambda_{1,1}\epsilon,\cdots, \lambda_1+\lambda_{1,k_1}\epsilon, \lambda_2+\lambda_{2,1}\epsilon,\cdots, \lambda_r+\lambda_{r,k_r}\epsilon\right).$$
Then we have $U^*AU = \Sigma$, as required.  Letting $U = (\vu_{1,1}, \cdots, \vu_{1, k_1}, \cdots, \vu_{r, k_r})$, we have (\ref{eee2}). Thus, $\lambda_i+\lambda_{i,j}\epsilon$ are eigenvalues of $A$ with $\vu_{i,j}$ as the corresponding eigenvectors, for $j = 1, \cdots, k_i$ and $i=1, \cdots, r$.   The uniqueness follows from Theorem \ref{t2.1}.

\end{proof}

With Theorems \ref{t3.3} and \ref{t4.1},  we have the following theorem.

\begin{Thm} \label{t4.4}
Suppose that $A \in {\mathbb {DQ}}^{n \times n}$ is Hermitian.  Then $A$ has exactly $n$ eigenvalues, which are all dual numbers.  There are also $n$ eigenvectors, associated with these $n$ eigenvalues.   The Hermitian matrix $A$ is positive semidefinite or definite if and only if all of these eigenvalues are nonnegative, or positive and appreciable, respectively.
\end{Thm}

\section{Perfect Hermitian Matrices}

Suppose that $M \in {\mathbb Q}^{n \times n}$ is a positive semidefinite quaternion Hermitian matrix.   Then there is a positive semidefinite quaternion Hermitian matrix $L \in {\mathbb Q}^{n \times n}$ such that $M = L^2$.   This is easy to verify, as we may write $M = UDU^*$, where $U \in {\mathbb Q}^{n \times n}$ is a unitary quaternion matrix, and $D \in {\mathbb R}^{n \times n}$ is a diagonal matrix with real nonnegative diagonal entries.   Then letting $L = UD^{1 \over 2}U^*$, we have the desired result.

However, this does not work in general for a positive semidefinite dual quaternion Hermitian matrix $M \in {\mathbb {DQ}}^{n \times n}$.   Let $M = I_n\epsilon$, where $I_n$ is the $n \times n$ identity matrix.
One cannot find $L \in {\mathbb {DQ}}^{n \times n}$ such that $M = L^2$.

We call a positive semidefinite dual quaternion Hermitian matrix $M \in {\mathbb {DQ}}^{n \times n}$ a perfect Hermitian matrix if there is a positive semidefinite dual quaternion Hermitian matrix $L \in {\mathbb {DQ}}^{n \times n}$ such that $M = L^2$.   In this section, we show that $M = B^*B$ is a perfect Hermitian matrix for any $B \in {\mathbb {DQ}}^{m \times n}$.   This property plays a key role in establishing the singular value decomposition of $B$, in the next section.

\begin{Thm}\label{special-Herm} Suppose that $B \in {\mathbb {DQ}}^{m \times n}$ and $A = B^*B$. Then there exists a unitary matrix $U\in {\mathbb {DQ}}^{n \times n}$ such that
\begin{equation}\label{block-diag}
U^* A U = {\rm diag}(\lambda_1+\lambda_{1,1}\epsilon, \cdots, \lambda_1+\lambda_{1,k_1}\epsilon, \lambda_2+\lambda_{2,1}\epsilon, \cdots, \lambda_s+\lambda_{s,k_s}\epsilon,O),
\end{equation}
where $\lambda_1>\cdots >\lambda_s>0$, $\lambda_{1,1} \ge \cdots \ge \lambda_{1,k_1}$, $\cdots$, $\lambda_{s,1}\ge \cdots \ge \lambda_{s,k_s}$ are real numbers.    Counting possible multiplicities $\lambda_{i, j}$, the real numbers $\lambda_i$ and $\lambda_{i, j}$ for $i=1, \cdots, s$ and $j = 1, \cdots, k_i$ are uniquely determined.

\end{Thm}
\begin{proof} Since $A$ is a positive semidefinite dual quaternion Hermitian matrix, by virtue of Theorem \ref{t4.1}, together with Theorem \ref{t4.4}, we know that $A$ can be diagonalized by $U$ as defined in Theorem \ref{t4.1}, and $A$ has exactly $n$ eigenvalues which are all nonnegative dual numbers, say $\lambda_{i}+\lambda_{i,j}\epsilon$, $i=1,\cdots, r$, $j=1,\cdots, k_i$, and $\lambda_1>\cdots >\lambda_r\geq 0$, $\lambda_{1,1} \ge \cdots \ge \lambda_{1,k_1}$, $\cdots$, $\lambda_{r,1}\ge \cdots \ge \lambda_{r,k_r}$. It then suffices to show that if $\lambda_r =0$, then $\lambda_{r,j}=0$ for every $j=1,\cdots, k_r$.   Note that
\begin{equation}\label{5-0} \vu_{r, j}^* A \vu_{r, j} = \vu_{r, j}^* \lambda_{r,j} \vu_{r, j} \epsilon = \lambda_{r,j} \epsilon, ~~j=1,\cdots, k_r.
\end{equation}
Combining with
\begin{equation}\label{5-1}
\vu_{r, j}^* A \vu_{r, j}
=  \vu_{r, j}^* B_{st}^*B_{st} \vu_{r, j} + \vu_{r, j}^* \left(B_{st}^*B_{\I}+B_{\I}^*B_{st}\right) \vu_{r, j}\epsilon,
\end{equation}
we have
\begin{equation}\label{5-2}
\vu_{r, j}^* B_{st}^*B_{st} \vu_{r, j} = 0,
\end{equation}
and hence $B_{st} \vu_{r, j} = {\bf 0}$. This further leads to
\begin{eqnarray}
&   & \vu_{r, j}^* \left(B_{st}^*B_{\I}+B_{\I}^*B_{st}\right) \vu_{r, j} \nonumber\\
& = &\left(B_{st} \vu_{r, j}\right)^* B_\I \vu_{r, j}+ \vu_{r, j}^*B_{\I}^*\left(B_{st}\vu_{r, j}\right)\nonumber\\
& = & 0. \nonumber
\end{eqnarray}
By invoking equations \eqref{5-0} and \eqref{5-1}, we immediately get  $\lambda_{r,j}=0$ for every $j=1,\cdots, k_r$. This completes the proof.


\end{proof}

\begin{Cor}\label{cor5.2}
Suppose that $B \in {\mathbb {DQ}}^{m \times n}$ and $A = B^*B$.   Then $A$ is a perfect Hermitian matrix.
\end{Cor}
\begin{proof} In Theorem \ref{special-Herm}, let
$$L = U\left(
            \begin{array}{cccccc}
              \sqrt{\lambda_{1}} + { \lambda_{1,1} \over 2\sqrt{\lambda_{1}}} \epsilon &   &   & & &\\
              &  \ddots & &  & & \\
               & &\sqrt{\lambda_{1}} + { \lambda_{1,k_1} \over 2\sqrt{\lambda_{1}}} \epsilon    &   & &\\
                &  & & \ddots &  &  \\
                &   & &  &\sqrt{\lambda_{s}}+{ \lambda_{s,k_s}\over 2\sqrt{\lambda_{s}}} \epsilon &  \\
                &    &  &  &  & O
            \end{array}
          \right)U^*.$$
Then $L$ is Hermitian and positive semidefinite, and $L^2 = A$.
\end{proof}

\section{Singular Value Decomposition of a Dual Quaternion Matrix}

We now present the singular value decomposition of dual quaternion matrices.

\begin{Thm}\label{SVD}
Suppose that $B \in {\mathbb {DQ}}^{m \times n}$. Then there exists a dual quaternion unitary matrix $\hat{V} \in {\mathbb {DQ}}^{m \times m}$ and a dual quaternion unitary matrix $\hat{U} \in {\mathbb {DQ}}^{n \times n}$ such that
\begin{equation} \label{e20}
\hat{V}^*B\hat{U} = \begin{bmatrix} \Sigma_t & O  \\ O & O \end{bmatrix},
\end{equation}
where $\Sigma_t\in {\mathbb{D}}^{t\times t}$ is a diagonal matrix, taking the form $$\Sigma_t ={\rm diag}\left(\mu_1, \cdots, \mu_r, \cdots, \mu_t \right),$$
$r \le t \le \min \{ m , n \}$, $\mu_1 \ge \mu_2 \ge \cdots \ge \mu_r$ are positive appreciable dual numbers, and $\mu_{r+1} \ge \cdots \ge \mu_t$ are positive infinitesimal dual numbers.    Counting possible multiplicities of the diagonal entries, the form $\Sigma_t$ is unique.

\end{Thm}
\begin{proof}  Let $A = B^* B$. It follows from Theorem \ref{special-Herm} 
that there exists a unitary matrix $U \in {\mathbb {DQ}}^{n \times n}$ as defined in Theorem \ref{t4.1} such that $A$ can be diagonalized as in \eqref{block-diag}. Set $r = \sum\limits_{j=1}^s k_j$,
\begin{equation}\label{sigma-cons}
\Sigma_r
\equiv \left(
            \begin{array}{ccccc}
              \sigma_1+\sigma_{1,1}\epsilon &   &   & &  \\
              &  \ddots & &  &   \\
               & &  \sigma_1+\sigma_{1,k_1}\epsilon  &   &  \\
                &  & & \ddots    &  \\
                &   &  & &\sigma_s+\sigma_{s,k_s}\epsilon
            \end{array}
          \right)
\end{equation}
with $\sigma_i = \sqrt{\lambda_{i}}$, $\sigma_{i,j} = { \lambda_{i,j} \over 2\sqrt{\lambda_{i}}}$, $j=1,\cdots, k_i$, $i=1,\cdots, s$. Denote
$U_1 = U_{:,1:r}$ and $U_2 = U_{:, r+1:n}$. By direct calculations, we have
$$AU = [B^* B U_1~~B^*BU_2 ]= [U_1\Sigma_r^2~~O],$$
yielding
$$ U_1^*B^* B U_1 = \Sigma_r^2, ~~U_2^*B^* B U_2 = O.$$
Therefore, $B U_2 = M \epsilon$ with some quaternion matrix $M$.

Let $V_1 = BU_1\Sigma_r^{-1}\in {\mathbb{DQ}}^{m\times r}$.
Then
$$V_1^*V_1 = I_s,~~V_1^*BU_1 = V_1^*V_1\Sigma_r = \Sigma_r,~~ V_1^*BU_2 = \left(\Sigma_r^{-1}\right)^*U_1^*(B^*BU_2) =  O.$$
Take $V_2 \in {\mathbb {DQ}}^{m \times (m-r)}$ such that $V = (V_1, V_2)$ is a unitary matrix.  We see that
$$V_2^*BU_1= V_2^*V_1\Sigma_r = O,~~V_2^*BU_2 = V_2^*M\epsilon = G\epsilon,$$
where $G$ is an $(m-r) \times (n-r)$ quaternion matrix.   Thus,
\begin{eqnarray*}
V^*BU & = & \begin{bmatrix} V_1^*BU_1 & V_1^*BU_2 \\ V_2^*BU_1 & V_2^*BU_2 \end{bmatrix}\\
& = & \begin{bmatrix} \Sigma_r & O \\ O & G\epsilon \end{bmatrix}.
\end{eqnarray*}
Applying Theorem 7.2 in \cite{Zh97}, there exist unitary matrices $W_1\in {\mathbb{Q}}^{(m-r)\times (m-r)}$ and $W_2\in{\mathbb{Q}}^{(n-r)\times (n-r)}$ such that
$W_1^* G W_2 = D$, where $D\in {\mathbb{Q}}^{(m-r)\times (n-r)}$ with $D_{ij}=0$ for any $i\neq j$ and $D_{ii}\geq 0$ for each $i=1,\cdots, \min\{m-r,n-r\}$. Denote
\begin{equation}\label{UV_bar}
\hat{V} \equiv V\left(
                  \begin{array}{cc}
                    I_r &   \\
                      & W_1 \\
                  \end{array}
                \right), ~~\mathrm{and}~~\hat{U} \equiv U\left(
                  \begin{array}{cc}
                    I_r &   \\
                      & W_2 \\
                  \end{array}
                \right).
\end{equation}
It is obvious that both $\hat{V}$ and $\hat{U}$ are unitary and $\hat{V}^* B \hat{U} = \begin{bmatrix} \Sigma_r & O  \\ O & D\epsilon \end{bmatrix}$.   Then we have (\ref{e20}).  The uniqueness of $\Sigma_t$ follows from Theorems \ref{t4.1} and \ref{t5.1}.
This completes the proof.
\end{proof}

We call $\mu_1, \cdots, \mu_t$ and possibly $\mu_{t+1} = \cdots = \mu_{\min \{m, n \} }=0$, if $t < \min \{m, n\}$, the singular values of $B$, $t$ the rank of $B$, and $r$ the appreciable rank of $B$. Several properties on the rank and the appreciable rank of a dual quaternion matrix are stated as follows.

\begin{Prop}\label{Prop-rank-1}
Suppose that $B = B_{st}+B_\I\epsilon \in {\mathbb {DQ}}^{m \times n}$. Then the appreciable rank of $B$ is equal to the rank of $B_{st}$.
\end{Prop}
\begin{proof} As one can see from the proof of Theorem \ref{SVD}, the appreciable rank of $B$ is exactly the number of positive eigenvalues of $A=B^*B$. Together with Theorem \ref{t4.1} and Theorem \ref{special-Herm}, this number is exactly the number of real positive right eigenvalues of $A_{st}=B_{st}^*B_{st}$, which is indeed the rank of $B_{st}$ from Theorem 7.2 in \cite{Zh97}.
\end{proof}
\begin{Prop}\label{Prop-rank-2}
Suppose that $P \in {\mathbb {DQ}}^{m \times n}$ is partially unitary. Then both the appreciable rank of $P$ and the rank of $P$ are exactly $n$.
\end{Prop}
\begin{proof}
By the definition of partially unitary matrices, we have $P^* P = I_n$. It then follows from the proof of Proposition \ref{Prop-rank-1} that the appreciable rank of $P$ is indeed the rank of $P_{st}^* P_{st}=I_n$, which is exactly $n$. Since the rank of $P$ is no less than its appreciable rank, and no greater than $\min\{m,n\}$. Thus, the rank of $P$ is also $n$.
\end{proof}

\section{Final Remarks}

In this paper, we have studied eigenvalues of dual quaternion Hermitian matrices and singular values of general dual quaternion matrices.   Some questions may be further considered.

1. What is the rank theory of dual quaternion matrices?

2. If the entries of a dual quaternion matrix are all unit dual quaternions, what special properties does such a matrix have?

3. Is there a determinant theory for dual quaternion matrices?

4. Besides the dual quaternion matrix recovery problem, can we list more application problems to motivate our study?

Clearly, more problems on dual quaternion matrices are worth exploring.

\bigskip

{\bf Acknowledgment}   We are thankful to  Ran Gutin and Chen Ling for their comments.



\end{document}